\documentclass[reqno,12pt]{amsart}
\usepackage{amsmath,amssymb}
\usepackage{amsfonts}
\usepackage{amsthm}
\usepackage{mathrsfs}
\usepackage{bm}
\usepackage[top=25truemm,bottom=25truemm,left=30truemm,right=30truemm]{geometry}
\usepackage{version}
\usepackage{url}
\usepackage{extarrows}
\usepackage{color}
\usepackage[all]{xy}
\usepackage{tikz}
\usetikzlibrary{intersections,calc,arrows}
\usepackage{graphics}
\usepackage{ascmac}
\usepackage{mathtools}
\usepackage[colorlinks, linkcolor=blue, citecolor=red, urlcolor=cyan]{hyperref}
\usepackage{breakurl}
\usepackage{cleveref}
\usepackage{tikz-cd}
\usepackage{version}

\numberwithin{equation}{section}

\newtheorem{thm}{Theorem}[section]

\newtheorem{lem}[thm]{Lemma}

\newtheorem*{introthm}{Theorem}

\theoremstyle{definition}
\newtheorem{dfn}[thm]{Definition}

\newtheorem{rem}[thm]{Remark}

\crefname{thm}{Theorem}{Theorems}
\crefname{cor}{Corollary}{Corollaries}
\crefname{lem}{Lemma}{Lemmas}
\crefname{prop}{Proposition}{Propositions}
\crefname{dfn}{Definition}{Definitions}
\crefname{ex}{Example}{Examples}
\crefname{claim}{Claim}{Claims}
\crefname{conj}{Conjecture}{Conjectures}
\crefname{rem}{Remark}{Remarks}
\crefname{figure}{Figure}{Figures}
\crefname{section}{Section}{Sections}
\crefname{subsection}{Section}{Sections}
\crefname{appendix}{Appendix}{Appendices}
\crefname{assum}{Assumption}{Assumptions}
\crefname{conv}{Notation}{Notations}

\crefname{introthm}{Theorem}{Theorems}
\crefname{introcor}{Corollary}{Corollaries}
\crefname{introconj}{Conjecture}{Conjectures}

\DeclareMathOperator{\wt}{wt}

\DeclareMathOperator{\Mat}{Mat}

\DeclareMathOperator{\Lie}{Lie}

\DeclareMathOperator{\ad}{ad}

\newcommand{\C}{\mathbb{C}}

\newcommand{\Z}{\mathbb{Z}}

\newcommand{\fg}{\mathfrak{g}}
\newcommand{\fh}{\mathfrak{h}}

\newcommand{\op}{\mathrm{op}}
\newcommand{\rfr}{\mathrm{fr}}
\newcommand{\rex}{\mathrm{ex}}

\newcommand{\tw}{\widetilde{w}}

\newcommand{\dprod}{\displaystyle\prod^{\longrightarrow}}
\newcommand{\rdprod}{\displaystyle\prod^{\longleftarrow}}
\begin{document}
\title[A cluster structure of the coordinate ring]
{A note on a cluster structure of the coordinate ring of a simple algebraic group}

\author[Hironori Oya]{Hironori Oya}
\address{Department of Mathematics, Institute of Science Tokyo, 2-12-1 Ookayama, Meguro-ku, Tokyo, 152-8551, Japan}
\email{hoya@math.titech.ac.jp}

\date{\today}

\begin{abstract}
We show that the coordinate ring of a simply-connected simple algebraic group $G$ over the complex number field coincides with the Berenstein--Fomin--Zelevinsky cluster algebra and its upper cluster algebra, at least when $G$ is not of type $F_4$.
\end{abstract}
\maketitle
\setcounter{tocdepth}{1}
\section{Introduction}
In the foundational paper \cite{CAI} of cluster algebras, Fomin and Zelevinsky conjectured that the coordinate rings of several varieties related to a simply-connected simple algebraic group $G$ have natural structures of cluster algebra. This conjecture has been widely investigated and verified. See, for example, \cite{CAI,CAII,CAIII,Sco,GLS08,GLS11,Wil,Lec,SSBW,FWZ,GY,SW21,GYmemo,SSB,CGGLSS}. 

As for the coordinate ring $\C[G]$ of $G$ itself, $\C[SL_2]$ is mentioned as a typical example of cluster algebra in \cite{CAI}. In \cite{FWZ}, it is shown that $\C[SL_n]$ is a cluster algebra for $n\geq 2$. However, in the literature, there does not seem to be a proof of the fact that $\C[G]$ has a cluster structure for a general simply-connected simple algebraic group $G$.

The goal of this paper is to provide a proof of this fact in the case when $G$ is not of type $F_4$. The precise statement is the following. 
\begin{introthm}[{=\cref{t:main}}]
Let $G$ be a simply-connected simple algebraic group over $\C$ which is not of type $F_4$. Then 
\[
\C[G]=\mathscr{A}=\mathscr{U},
\]
where $\mathscr{A}$ is the cluster algebra associated with the Berenstein--Fomin--Zelevinsky seed for the open double Bruhat cell in $G$, and $\mathscr{U}$ is the corresponding upper cluster algebra. Here we do not add the inverse of the frozen variables in the definition of $\mathscr{A}$ and $\mathscr{U}$. 
\end{introthm}

In \cite[Theorem 2.10]{CAIII}, it is shown that $\mathscr{U}^{\circ}$ coincides with  the coordinate ring $\C[G^{w_0, w_0}]$ of the open double Bruhat cell $G^{w_0, w_0}$ in $G$, and in \cite{GYmemo}, \cite[Theorem 4.13]{SW21}, it is shown that $\mathscr{U}^{\circ}=\mathscr{A}^{\circ}$, including the case of type $F_4$. Here $\mathscr{A}^{\circ}$ and $\mathscr{U}^{\circ}$ denote the cluster and upper cluster algebras obtained from $\mathscr{A}$ and $\mathscr{U}$, respectively, by adding the inverses of the frozen variables. However, the fact that  $\C[G^{w_0, w_0}]=\mathscr{A}^{\circ}=\mathscr{U}^{\circ}$ does not immediately imply $\C[G]=\mathscr{A}=\mathscr{U}$. Indeed, all the equalities $\C[G]=\mathscr{A}$, $\mathscr{A}=\mathscr{U}$, and $\C[G]=\mathscr{U}$ seem to be non-trivial in general, while the inclusion $\mathscr{A}\subset \mathscr{U}$ follows from the general fact, known as the Laurent phenomenon, of the theory of cluster algebras (\cite[Theorem 3.1]{CAI}, \cite[Proposition 11.2]{CAII}). The delicate relation between $\mathscr{A}^{\circ}=\mathscr{U}^{\circ}$ and $\mathscr{A}=\mathscr{U}$ is discussed in \cite{BMS}. The third equality $\C[G]=\mathscr{U}$ was announced in \cite[Claim 7.6]{Qin} for an arbitrary $G$  without proof, and the proof is now provided in a very recent preprint \cite[Theorem B.1]{QY}\footnote{While the author was preparing this paper, he was informed that a proof of $\C[G] = \mathscr{U}$ is given in \cite{QY} using a different method. The author is grateful to Fan Qin and Milen Yakimov for letting him know about their ongoing work.}.

Here we make a remark on an application of the main result of this paper. When we study the algebraic structure on the ring of regular functions on some variety $X$ (over $\C$), a morphism $X\to G$ often plays an important role, since it induces an algebra homomorphism $\C[G]\to \C[X]$ via pullback. In such a situation, the cluster structure on $\C[G]$ is useful for studying the cluster structure on $\C[X]$. For example\footnote{In a recent work \cite{FL}, Francone and Leclerc study the cluster structure on the coordinate ring of the scheme $B(G, c)$, called $(G, c)$-bands ($c$ is a Coxeter element), by using the projection maps $B(G, c)\to G$ and the result of this paper.}, when $X$ is the moduli space $\mathcal{A}_{G, \Sigma}^{\times}$ of generically decorated twisted $G$-local systems on a marked surface $\Sigma$ with boundaries, one can define a family of morphisms $\mathcal{A}_{G, \Sigma}^{\times}\to G$, called Wilson lines \cite{IOS23}. These morphisms are essentially used in the study of the cluster structure on $\C[\mathcal{A}_{G, \Sigma}^{\times}]$. Indeed, in \cite{IOS23}, $\C[G]=\mathscr{A}$ was implicitly proved in the case when $G$ is not of type $E_8, F_4, G_2$, and it is precisely the assumption on $G$ in \cite[Theorem 1]{IOS23}. By the main result of this paper, we can show that the result parallel to \cite[Theorem 1]{IOS23} holds also in the case of type $E_8$ and $G_2$. The detailed discussion will appear in a separate publication. Note that we cannot restrict the Wilson lines to the morphisms $\mathcal{A}_{G, \Sigma}^{\times}\to G^{w_0, w_0}$, hence the corresponding morphisms $\C[G]\to \C[\mathcal{A}_{G, \Sigma}^{\times}]$ cannot be extended to 
$\C[G^{w_0, w_0}]\to \C[\mathcal{A}_{G, \Sigma}^{\times}]$ in general.

In our proof of the main theorem, the result \cite[Theorem 1.4]{GLS13} (see \cref{t:GLS}) is essential. It reduces the proof of the equalities in the main theorem to the proof of the inclusion $\C[G]\subset \mathscr{A}$. Here we do not use the assumption that $G$ is not of type $F_4$. Then, under the assumption, we prove $\C[G]\subset \mathscr{A}$ by showing that $\C[G]$ is generated by the generalized minors. This fact is shown in \cite[Proposition 2.2]{IOS23} when $G$ is of type $A_n$, $B_n$, $C_n$, $D_n$, $E_6$ or $E_7$. Therefore, the main contribution of this paper is to show this property in the cases of type $E_8$ and $G_2$. 
\cref{r:F4} explains why the techniques developed here to deal with the $E_8$ and $G_2$ cases do not work for $F_4$.

\subsection*{Acknowledgments}
The author would like to thank Tsukasa Ishibashi and Linhui Shen for numerous helpful discussions and comments. He is greatly indebted to Bernard Leclerc, whose question served as a strong motivation for writing this paper. He also wishes to express his thanks to Fan Qin and Milen Yakimov for kindly explaining their works.
He is grateful to Luca Francone and Bernard Leclerc for letting him know their ongoing works. 
He would like to thank the anonymous referees for their careful reading of the manuscript and for suggesting many improvements. 

This work was supported by JSPS KAKENHI Grant-in-Aid for Early-Career Scientists, Grant Number 23K12950.

\section{Prerequisites}\label{s:pre}

\subsection{Simple algebraic groups}
Let $\fg$ be a complex finite dimensional simple Lie algebra associated with a Cartan matrix $A=(a_{st})_{s, t\in S}$. That is, $\fg$ is the complex Lie algebra generated by $\{e_s, f_s, h_s\mid s\in S\}$ subject to the following relations: for $s, t\in S$, 
\begin{itemize}
	\item[(i)] $[h_s, h_t]=0$,
	\item[(ii)] $[h_s, e_{t}]=a_{st}e_{t}$, $[h_s, f_{t}]=-a_{st}f_{t}$, 
	\item[(iii)] $[e_{s},f_{t}]=\delta_{st}h_s$, 
	\item[(iv)] ${(\ad e_{s})^{1-a_{st}}(e_{t})=0}$ and ${(\mathrm{ad} f_{s})^{1-a_{st}}(f_{t})=0}$ for $s\neq t$. 
    
    \noindent Here $(\ad x)(y):=[x, y]$ for $x, y\in \fg$.
\end{itemize} 
We set $S=\{1,\dots, n\}$, following the labeling in \cite[Chapter 4]{Kac}. Let $\fh:=\sum_{s\in S}\mathbb{C}h_s$ be a Cartan subalgebra of $\fg$. Define the \emph{weight lattice} $P$ and the monoid of \emph{dominant integral weights} $P_+$ by  
\begin{align*}
&P:=\{\mu\in \fh^{\ast}\mid \langle h_s, \mu\rangle\in \mathbb{Z}\text{ for all } s\in S\},\\
&P_+:=\{\lambda\in P\mid \langle h_s, \lambda \rangle\geq 0\text{ for all } s\in S\}. 
\end{align*}
Denote by $\varpi_s\in P_+$ the \emph{fundamental weight} satisfying $\langle h_t, \varpi_s\rangle=\delta_{ts}$ for $t\in S$. Denote by $\Pi=\{\alpha_s\}_{s\in S}, \Phi$, and $\Phi^+$ the set of \emph{simple roots}, the \emph{root system} of $\fg$, and the set of \emph{positive roots}, respectively. We regard $\Pi, \Phi, \Phi^+$ as subsets of $P$. Set $Q\coloneqq \sum_{s\in S}\Z \alpha_{s}$ and $Q_+\coloneqq \sum_{s\in S}\Z_{\geq 0} \alpha_{s}$. There is a partial ordering on $P$ defined by $\mu\leq \lambda$ if and only if $\lambda-\mu\in Q_+$. We write $\mu <\lambda$ if $\mu\leq \lambda$ and $\mu\neq \lambda$. 

Let $W$ be the Weyl group of $\fg$, generated by the simple reflections $\{r_s \mid s\in S\}$. The identity element of $W$ is denoted by $e$. 
The Weyl group $W$ acts on $P$ by 
\[
r_s(\mu)=\mu-\langle h_s, \mu\rangle\alpha_s
\]
for $\mu\in P$ and $s\in S$. For $w\in W$, let
\[
\ell(w)\coloneqq\min \{\ell\in \Z_{\geq 0} \mid w=r_{s_1}\cdots r_{s_\ell}\},
\]
be the \emph{length} of $w$. It is known that $W$ has a unique element of maximum length, which is denoted by $w_0$.  
For $w\in W$, set 
\[
R(w)\coloneqq \{(s_1,\dots, s_{\ell(w)})\in S^{\ell(w)}\mid w=r_{s_1}\cdots r_{s_{\ell(w)}}\}.
\]
(Notice that $R(e)$ is the empty set $\varnothing$.) An element of $R(w)$ is called a \emph{reduced word} of $w$. 

For $(u, v)\in W\times W$, $(s_1,\dots, s_{\ell(u)+\ell(v)})\in (S\sqcup -S)^{\ell(u)+\ell(v)}$ is called a \emph{double reduced word} of $(u, v)$ if $(-s_{i_1},\dots, -s_{i_k})\in R(u)$ and $(s_{j_1},\dots, s_{j_m})\in R(v)$, where 
\begin{itemize}
    \item $\{i\mid s_{i}<0,\ 1\leq i\leq \ell(u)+\ell(v)\}=\{i_1,\dots, i_k\}$ and $i_1<\cdots< i_k$,
    \item $\{j\mid s_{j}>0,\ 1\leq j\leq \ell(u)+\ell(v)\}=\{j_1,\dots, j_m\}$ and $j_1<\cdots< j_m$. 
\end{itemize}
The set of double reduced words of $(u, v)$ is written as $R(u, v)$. 

Let $G$ be a simply-connected simple algebraic group over $\mathbb{C}$ whose Lie algebra $\Lie(G)$ is isomorphic to $\mathfrak{g}$, and fix a maximal torus $H$ of $G$. We consider an isomorphism $\Lie(G)\xrightarrow[]{\sim}\mathfrak{g}$ which maps the Lie algebra of $H$ to $\fh$. For $s\in S$, we write the root subgroup of $G$ corresponding to $\alpha_s$ (resp.~$-\alpha_s$) as $U_s^+$ (resp.~$U_s^-$). We take additive one-parameter subgroups $x_{s}\colon \C\xrightarrow{\sim} U_s^+$ and $y_{s}\colon \C\xrightarrow{\sim} U_s^-$ so that there exists a homomorphism $\varphi_s\colon SL_2(\mathbb{C})\to G$ of algebraic groups, satisfying 
\begin{align*}
\varphi_s\left(\begin{pmatrix}
1&t\\
0&1
\end{pmatrix}\right)&=x_s(t),&
\varphi_s\left(\begin{pmatrix}
1&0\\
t&1
\end{pmatrix}\right)&=y_s(t). 
\end{align*}

For $s\in S$, set 
\[
\overline{r}_{s}\coloneqq \varphi_s\left(\begin{pmatrix}
0&-1\\
1&0
\end{pmatrix}\right).
\]
For $w\in W$, take $(s_1,\dots, s_{\ell})\in R(w)$ and set $\overline{w}:=\overline{r}_{s_1}\cdots \overline{r}_{s_{\ell}}\in G$. It is known that $\overline{w}$ does not depend on the choice of reduced words of $w$. 

Let $V$ be a finite dimensional $\fg$-module. There is a unique $G$-module structure on $V$ satisfying 
\begin{align*}
x_s(t)\cdot v&=\sum_{k=0}^\infty \frac{t^k}{k!}e_s^k\cdot v&
y_s(t)\cdot v&=\sum_{k=0}^\infty \frac{t^k}{k!}f_s^k\cdot v
\end{align*}
for $s\in S, t\in\C$ and $v\in V$. Note that $e_s$ and $f_s$ act nilpotently on $V$. In the following, we tacitly regard finite dimensional $\fg$-modules as $G$-modules via this correspondence. 

Let $V$ be a finite dimensional $\fg$-module. Its dual space $V^{\ast}$ is considered as a $\fg$-module via
\[
\langle x\cdot f, v\rangle\coloneqq -\langle f, x\cdot v\rangle
\]
for $x\in \fg$, $f\in V^{\ast}$ and $v\in V$. For $\mu \in P$, we set
	\[
	V_{\mu}:= \{v \in V \mid h\cdot v=\langle \mu, h \rangle v\ \text{\ for\ all\ } h \in \fh\}.
	\]
When $V_{\mu}\neq \{0\}$, $\mu$ is called a \emph{weight} of $V$, and for $v\in V_{\mu}$, we write $\wt v\coloneqq \mu$.

Let $\C[G]$ be the coordinate ring of $G$. For a finite dimensional $\fg$-module $V$, $f\in V^{\ast}$ and $v\in V$, we can define an element $C^V(f, v)\in \C[G]$, called a \emph{matrix coefficient} of $V$, by  
\[
g\mapsto \langle f, g\cdot v\rangle
\]
for $g\in G$. Note that, if $V$ and $V'$ are finite dimensional $\fg$-modules, then 
\begin{align*}
C^V(f, v)C^{V'}(f', v')=C^{V\otimes V'}(f\otimes f', v\otimes v')\label{eq:mult}
\end{align*}
for $f\in V^{\ast}, v\in V$ and $f'\in (V^{\prime})^\ast, v'\in V'$. Here $f\otimes f'\in V^{\ast}\otimes (V^{\prime})^\ast$ is considered as the element of $(V\otimes V')^{\ast}$ given by $v\otimes v'\mapsto \langle f, v\rangle\langle f', v'\rangle$ for $v\in V$ and $v'\in V'$. Note that this correspondence gives an isomorphism $V^{\ast}\otimes (V^{\prime})^\ast\simeq (V\otimes V')^{\ast}$ of $\fg$-modules, and we always identify these two spaces by this isomorphism. 

For $\lambda\in P_{+}$, let $V(\lambda)$ be the irreducible $\fg$-module of highest weight $\lambda$. Then we have an isomorphism of $G\times G$-modules 
\[
\bigoplus_{\lambda\in P_+}V(\lambda)^{\ast}\boxtimes V(\lambda)\xrightarrow{\sim}\C[G],\ 
f\otimes v\mapsto C^{V(\lambda)}(f, v)\ (f\in V(\lambda)^{\ast}, v\in V(\lambda)),
\]
which is an algebraic version of the Peter--Weyl theorem. See, for example, \cite[Theorem 4.2.7]{GW}.

We fix a highest weight vector of $V(\lambda)$ and denote it by $v_{\lambda}$. For $w\in W$, set 
\begin{align*}
v_{w\lambda}\coloneqq \overline{w}\cdot v_{\lambda}\in V(\lambda)_{w\lambda}, 
\end{align*}
and define $f_{w\lambda}\in (V(\lambda)^{\ast})_{-w\lambda}$ by $\langle f_{w\lambda}, v_{w\lambda}\rangle=1$. 

For $\lambda\in P_+$ and $w,w'\in W$, we write 
\[
\Delta_{w\lambda, w'\lambda}\coloneqq C^{V(\lambda)}(f_{w\lambda},v_{w'\lambda})
\]
and call it a \emph{generalized minor}. Note that we have 
\[
\Delta_{w\lambda, w'\lambda}\Delta_{w\lambda', w'\lambda'}=\Delta_{w(\lambda+\lambda'), w'(\lambda+\lambda')}
\]
for $\lambda, \lambda'\in P_+$ and $w,w'\in W$.
\subsection{Cluster algebras} In this subsection, we recall a result of \cite{GLS13} on cluster algebras which is essential in this paper (see \cref{t:GLS}).  
We do not review the definition of cluster algebras of geometric type, but we always follow the definitions in \cite[Section 1]{GLS13} unless otherwise specified. See \cite{GLS13} for all missing definitions. 

Let $J$ be a finite set with $|J|>1$, and $J_{\rex}\subset J$ be a nonempty subset. Write $J_{\rfr}\coloneqq J\setminus J_{\rex}$. Take a $J\times J_{\rex}$-matrix $B=(b_{ij})_{i\in J, j\in J_{\rex}}\in \Mat_{J\times J_{\rex}}(\Z)$ with integer entries, whose principal part $B^{\circ}=(b_{ij})_{i, j\in J_{\rex}}$ is skew-symetrizable and connected. 
Then we can consider a seed $(\bm{x}=(x_i)_{i\in J}, B)$ of the ambient field $\mathscr{F}$. We recall that the field $\mathscr{F}$ is assumed to be freely generated over $\C$ by the elements $x_i$ for $i\in J$. 
Two seeds $(\bm{x}', B')$ and $(\bm{x}'', B'')$ are said to be \emph{mutation equivalent}, and written as $(\bm{x}', B')\sim (\bm{x}'', B'')$, if there exist $i_1,\dots, i_t\in J_{\rex}$ satisfying 
\[
\mu_{i_t}\cdots \mu_{i_1}(\bm{x}', B')=(\bm{x}'', B'').
\]
Here $\mu_i$ denotes the mutation of a seed at the vertex $i$. The \emph{cluster algebra} $\mathscr{A}(\bm{x}, B)$ associated to the seed $(\bm{x}, B)$ is defined as the subalgebra of $\mathscr{F}$ generated over the ring $\C[x_i\mid i\in J_{\rfr}]$ by the set of \emph{cluster variables} 
\[
\mathcal{X}_{(\bm{x}, B)}\coloneqq \bigcup_{(\bm{x}, B)\sim (\bm{x}'=(x'_i)_{i\in J}, B')}\{x'_i\mid i\in J_{\rex}\}.
\]
The \emph{upper cluster algebra} $\mathscr{U}(\bm{x}, B)$ associated to $(\bm{x}, B)$ is defined as the subalgebra of $\mathscr{F}$
\[
\mathscr{U}(\bm{x}, B)\coloneqq \bigcap_{(\bm{x}, B)\sim (\bm{x}'=(x'_i)_{i\in J}, B')}\C[x_i^{\prime \pm 1}, x_j\mid i\in J_{\rex}, j\in J_{\rfr}].
\]
We have the inclusion $\mathscr{A}(\bm{x}, B)\subset \mathscr{U}(\bm{x}, B)$, which is known as the \emph{Laurent phenomenon} (\cite[Theorem 3.1]{CAI}, \cite[Proposition 11.2]{CAII}). 
The elements $x_i$ ($i\in J_{\rfr}$) are called \emph{frozen variables}. 
We remark that the inverses of the frozen variables are not added in the definition of the cluster algebras in this paper.

When $(\bm{x}', B')\sim (\bm{x}, B)$, $\bm{x}'$ is called a \emph{cluster} of $\mathscr{A}(\bm{x}, B)$. Two clusters $\bm{x}'=(x'_i)_{i\in J}$ and $\bm{x}''=(x''_i)_{i\in J}$ of $\mathscr{A}(\bm{x}, B)$ are said to be \emph{disjoint} if $\{x'_i\mid i\in J_{\rex}\}\cap \{x''_i\mid i\in J_{\rex}\}=\varnothing$. 

The following is one of the main results of \cite{GLS13}. 
\begin{thm}[{\cite[Theorem 1.4]{GLS13}}]\label{t:GLS}
    Let $\bm{x}'=(x'_i)_{i\in J}$ and $\bm{x}''=(x''_i)_{i\in J}$ be disjoint clusters of  $\mathscr{A}(\bm{x}, B)$. Assume that there exists a subalgebra $\mathscr{O}$ of $\mathscr{A}(\bm{x}, B)$ such that 
    \begin{itemize}
        \item[(i)] $\mathscr{O}$ is a unique factorization domain, 
        \item[(ii)] $\{x'_i, x''_i\mid i\in J_{\rex}\}\cup \{x_i\mid i\in J_{\rfr}\}\subset \mathscr{O}$.        
    \end{itemize}
    Then $\mathscr{O}=\mathscr{A}(\bm{x}, B)=\mathscr{U}(\bm{x}, B)$. 
\end{thm}

\subsection{Berenstein--Fomin--Zelevinsky seeds}
Let $(u, v)\in W\times W$. In \cite[Theorem 2.10]{CAIII}, it is shown that there is a structure of upper cluster algebra with invertible frozen variables on the coordinate ring $\C[G^{u, v}]$ of the double Bruhat cell $G^{u, v}=B^+\overline{u}B^+\cap B^-\overline{v}B^-$. Here $B^+$ and $B^-$ denote the Borel subgroups containing $H$ corresponding to $\Phi^+$ and $-\Phi^+$, respectively.

We consider the case when $u=v=w_0$. Then $G^{w_0, w_0}$ is an open subvariety of $G$. In \cite{CAIII}, the authors construct a seed $(\bm{x}_{\bm{s}}, B_{\bm{s}})$ of the field $\mathscr{F}=\C(G^{w_0, w_0})=\C(G)$ of rational functions on $G^{w_0, w_0}$ (hence, on $G$) associated with $\bm{s}\in R(w_0, w_0)$. In this paper, we do not review the definition of $B_{\bm{s}}$, but recall the definition of $\bm{x}_{\bm{s}}$. 

Let $N\coloneqq 2\ell (w_0)$ and $\bm{s}=(s_1,\dots, s_N)\in R(w_0, w_0)$. Set
\[
J=-S\cup\{1,\dots, N\}=\{-n,\dots, -1, 1,\dots, N\}. 
\]
For $k\in J$, set 
\[
k^+=k^+_{\bm{s}}\coloneqq  \min(\{j\mid j>k,\  |s_j|=|s_k| \} \cup \{ N + 1 \}),
\]
and 
\[
J_{\rfr}\coloneqq -S\cup \{k\in J\mid k^+=N+1\},\quad 
J_{\rex}\coloneqq J\setminus J_{\rfr}.
\]
For $k\in J$, we set 
\begin{align*}
    \Delta(k; \bm{s})\coloneqq \Delta_{\tw_{\leq k}\varpi_{|s_k|}, \tw_{>k}\varpi_{|s_k|}}.
\end{align*}
Here
\begin{align*}
\tw_{\leq k}&=\tw_{\leq k}(\bm{s})\coloneqq \dprod_{\ell; 1\leq \ell\leq k, s_{\ell}<0}r_{-s_{\ell}}\quad (k>0),&
\tw_{\leq k}&=e\quad (k<0),\\
\tw_{> k}&=\tw_{> k}(\bm{s})\coloneqq \rdprod_{\ell; k< \ell\leq N, s_{\ell}>0}r_{s_{\ell}}\quad (k>0),&
\tw_{>k}&=w_0\quad (k<0),
\end{align*}
and $\dprod$ (resp.~$\rdprod$) denotes the ordered product, taken in increasing (resp.~decreasing) order of the index $\ell$. Then 
\[
\bm{x}_{\bm{s}}=(\Delta(k; \bm{s}))_{k\in J}. 
\]
Note that the set of the frozen variables is $\{\Delta_{\varpi_s, w_0\varpi_s}, \Delta_{w_0\varpi_s, \varpi_s}\mid  s\in S\}$. 

\begin{lem}\label{l:genminor}
    For any $w, w'\in W$ and $s\in S$, there exist $\bm{s}\in R(w_0, w_0)$ and $k\in J$ such that 
    \[
    \Delta(k; \bm{s})=\Delta_{w\varpi_s, w'\varpi_s}.
    \]
\end{lem}
\begin{proof}
Take $(s_{i_1},\dots, s_{i_{N/2}}), (s_{j_1},\dots, s_{j_{N/2}})\in R(w_0)$ such that $\bm{s}_w\coloneqq (s_{i_1},\dots, s_{i_{\ell(w)}})\in R(w)$ with $s_{i_{\ell(w)}}=s$ and $\bm{s}_{w'}\coloneqq (s_{j_1},\dots, s_{j_{\ell(w')}})\in R(w')$ with $s_{j_{\ell(w')}}=s$. Here, if there are no reduced words of $w$ ending with $s$, then we may replace $w$ with $w''\in W$ such that 
\begin{itemize}
    \item $w\varpi_s=w''\varpi_s$, and
    \item $w''$ has a reduced word ending with $s$, or $w''=e$, 
\end{itemize}
and do the same  procedure for $w'$. Moreover, if $w=e$ (resp.~$w'=e$), there is no condition on $(s_{i_1},\dots, s_{i_{N/2}})$ (resp.~$(s_{j_1},\dots, s_{j_{N/2}})$). Set 
\begin{align*}
    &\bm{s}_{w'}^{\op}\coloneqq (s_{j_{\ell(w')}},\dots, s_{j_1})\in R((w')^{-1}),\\
    &\bm{s}_{w}'\coloneqq (s_{i_{\ell(w)+1}},\dots, s_{i_{N/2}})\in R(w^{-1}w_0),\\
    &\bm{s}_{w'}^{\prime \op}\coloneqq (s_{j_{N/2}},\dots, s_{j_{\ell(w')+1}})\in R(w_0w').
\end{align*}
Then, 
\begin{align*}
\bm{s}&\coloneqq (\bm{s}_{w'}^{\prime \op}, -\bm{s}_w, \bm{s}_{w'}^{\op}, -\bm{s}_{w}')\\
&=(\underbrace{s_{j_{N/2}},\dots, s_{j_{\ell(w')+1}}}, \underbrace{-s_{i_1},\dots, -s_{i_{\ell(w)}}}, \underbrace{s_{j_{\ell(w')}},\dots, s_{j_{1}}}, \underbrace{-s_{i_{\ell(w)+1}},\dots, -s_{i_{N/2}}})
\end{align*}
is an element of $R(w_0, w_0)$, and by definition, we have 
\[
\Delta(\ell(w)-\ell(w')+N/2; \bm{s})=\Delta_{w\varpi_s, w'\varpi_s}.
\]
\end{proof}
\begin{thm}[{\cite[Remark 2.14]{CAIII}, \cite[Section 3]{SW21}}]\label{p:mutequiv}
    For any $\bm{s}, \bm{s}'\in R(w_0, w_0)$, $(\bm{x}_{\bm{s}}, B_{\bm{s}})\sim (\bm{x}_{\bm{s}'}, B_{\bm{s}'})$.
\end{thm}
\begin{lem}\label{l:disjoint}
    For $(s_{i_1},\dots, s_{i_{N/2}}), (s_{j_1},\dots, s_{j_{N/2}})\in R(w_0)$, set 
    \[
    \bm{s}_1\coloneqq (-s_{i_1},\dots, -s_{i_{N/2}}, s_{j_1},\dots, s_{j_{N/2}}),\quad
    \bm{s}_2\coloneqq (s_{i_1},\dots, s_{i_{N/2}}, -s_{j_1},\dots, -s_{j_{N/2}}).
    \]
    Then the clusters $\bm{x}_{\bm{s}_1}$ and $\bm{x}_{\bm{s}_2}$ are disjoint. 
\end{lem}
\begin{proof}
The cluster variables of $\bm{x}_{\bm{s}_1}$ are of one of the following forms:
\begin{itemize}
    \item $\Delta_{w\varpi_s, w_0\varpi_s}$ for $w\in W, s\in S$ with $w\varpi_s<\varpi_s$.
    \item $\Delta_{w_0\varpi_s, w\varpi_s}$ for $w\in W, s\in S$ with $w\varpi_s<\varpi_s$.    
\end{itemize}
Note that neither $w_0\varpi_s$ nor $w\varpi_s$ with $w\varpi_s<\varpi_s$ is dominant. 

On the other hand, the cluster variables of $\bm{x}_{\bm{s}_2}$ are of one of the following forms:
\begin{itemize}
    \item $\Delta_{\varpi_s, w\varpi_s}$ for $w\in W, s\in S$.
    \item $\Delta_{w\varpi_s, \varpi_s}$ for $w\in W, s\in S$. 
\end{itemize}
Since $\varpi_s$ is dominant, every cluster variable of $\bm{x}_{\bm{s}_1}$ does not belong to $\bm{x}_{\bm{s}_2}$. 
\end{proof}
\section{Main result}\label{s:main}
\subsection{Main result}
The following is the main result of this paper. 
\begin{thm}\label{t:main}
Assume that $G$ is not of type $F_4$. Then, for $\bm{s}\in R(w_0, w_0)$, 
\[
\C[G]=\mathscr{A}(\bm{x}_{\bm{s}}, B_{\bm{s}})=\mathscr{U}(\bm{x}_{\bm{s}}, B_{\bm{s}}).
\]
Here recall that the frozen variables are non-invertible in $\mathscr{A}$ and $\mathscr{U}$. 
\end{thm}
\begin{proof}
We apply \cref{t:GLS} by setting $\mathscr{O}\coloneqq \C[G]$. Indeed, it is known that $\C[G]$ is a unique factorization domain (see \cite[Theorem 3]{KP83} and the references therein). Moreover, $\C[G]$ satisfies the condition (ii) in \cref{t:GLS} by \cref{p:mutequiv} and \cref{l:disjoint}. Therefore, it remains to show that $\C[G]$ is a subalgebra of $\mathscr{A}(\bm{x}_{\bm{s}}, B_{\bm{s}})$. In other words, we need to show that $\mathscr{A}(\bm{x}_{\bm{s}}, B_{\bm{s}})$ contains generators of $\C[G]$. Hence, the proof is completed by the following theorem (\cref{t:mainminor}), together with \cref{l:genminor} and \cref{p:mutequiv}. 
\end{proof}
\begin{thm}\label{t:mainminor}
Assume that $G$ is not of type $F_4$. Then $\C[G]$ is generated by the generalized minors. 
\end{thm}
\begin{rem}
    We do not state that the statements of \cref{t:main,t:mainminor} do not hold in the case when $G$ is of type $F_4$. Indeed, $\C[G] = \mathscr{U}$ is proved for an arbitrary $G$ in a very recent preprint \cite[Theorem B.1]{QY} by a different method.  
\end{rem}
\cref{t:mainminor} is already shown in \cite[Proposition 2.2]{IOS23} when $G$ is of type $A_n$, $B_n$, $C_n$, $D_n$, $E_6$ or $E_7$. Indeed, in these cases, $G$ admits a faithful minuscule representation $V$. This representation gives a closed immersion $\rho\colon G\to SL(V)$ (see \cite[Corollary 1.13]{Brion} and \cite[Proposition 2.1]{IOS23}). Therefore, the coordinate ring $\C[G]$ is generated by the matrix coefficients of $V$, hence, by the generalized minors since $V$ is minuscule. See \cite[Section 2.2]{IOS23} for more details. 

Therefore, the remaining cases for the proof of \cref{t:mainminor} are the cases of type $E_8$ and $G_2$. The remainder of this paper is devoted to the proof of \cref{t:mainminor} in the cases of type $E_8$ and $G_2$.

\subsection{Proof of \cref{t:mainminor} in the cases of type $E_8$ and $G_2$} 
Although we are mainly interested in the simple Lie algebras of type $E_8$ and $G_2$, all the arguments in this subsection are valid under the assumption that  $\fg$ is a complex simple Lie algebra whose rank is greater than $1$. Therefore, we will work under this assumption. 

For $\mu\in P$ and a $\fg$-module $V$ with weight space decomposition, we write 
    \begin{align*}
    \wt(V)\coloneqq \{ \mu \in P \mid V_{\mu} \neq \{0\} \},\qquad
    V_{\neq \mu}\coloneqq \bigoplus_{\nu\in P, \nu\neq \mu}V_{\nu}.
    \end{align*}

\begin{dfn}
    A non-trivial irreducible $\fg$-module $V(\lambda)$ is said to be \emph{quasi-minuscule} if $\wt(V(\lambda))=W\cdot \lambda\cup\{0\}$. 
\end{dfn}
\begin{lem}\label{l:qminwt}
Let $V$ be a quasi-minuscule $\fg$-module. Then there exists $s\in S$ such that $\wt(V)\setminus\{0\}=W\cdot \alpha_s$. In particular, $\wt(V)\subset Q_+\cup (-Q_+)$.  
\end{lem}
\begin{proof}
    Since $V_0\neq \{0\}$ and $V$ is non-trivial, there exists $s\in S$ such that $e_s\cdot V_0\neq \{0\}$. Then $\alpha_s\in \wt (V)$, and by the definition of quasi-minuscule $\fg$-module, we have $\wt(V)=W\cdot \alpha_s\cup \{0\}$. The last assertion follows from $W\cdot \alpha_s\subset \Phi$. 
\end{proof}
\begin{lem}\label{l:0wtsp}
Let $V$ be a quasi-minuscule $\fg$-module. Then $\sum_{s\in S} f_se_s\cdot V_0=V_0$. 
\end{lem}
\begin{proof}
Since $V$ is a highest weight module, $V_0=\sum_{s\in S} f_s\cdot V_{\alpha_s}$. On the other hand, since $V$ is a lowest weight module, $V_{\alpha_s}=\sum_{t\in S}e_t\cdot V_{\alpha_s-\alpha_t}=e_s\cdot V_0$. Here the last equality follows from \cref{l:qminwt}. Therefore, $V_0=\sum_{s\in S} f_se_s\cdot V_{0}$. 
\end{proof}
\begin{lem}\label{l:Key}
Let $V$ be a quasi-minuscule $\fg$-module. Assume that there exists a surjective homomorphism $m\colon V\otimes V\to V$ of $\fg$-modules satisfying $m(V_0\otimes V_0)=\{0\}$. Then $m(V_{\neq 0}\otimes V_{\neq 0})=V$. 
\end{lem}
\begin{proof}
Let $\mu\in \wt (V)$. It suffices to show that 
\[
m\left(\bigoplus_{\nu\in P; 0\neq \nu\neq \mu}V_{\nu}\otimes V_{\mu-\nu}\right)=V_{\mu}. 
\]
When $\mu=0$, the condition $m(V_0\otimes V_0)=0$ and the surjectivity of $m$ imply 
\[
m\left(\bigoplus_{\nu\in P; \nu\neq 0}V_{\nu}\otimes V_{-\nu}\right)=m\left(\bigoplus_{\nu\in P}V_{\nu}\otimes V_{-\nu}\right)=V_0.
\]
Assume that $\mu\neq 0$. Then we have $\overline{w}\cdot V_{\mu}=V_{w(\mu)}$, and 
\begin{align*}
\overline{w}\cdot m\left(\bigoplus_{\nu\in P; 0\neq \nu\neq \mu}V_{\nu}\otimes V_{\mu-\nu}\right)&=m\left(\bigoplus_{\nu\in P; 0\neq \nu\neq \mu}\overline{w}\cdot V_{\nu}\otimes \overline{w}\cdot V_{\mu-\nu}\right)\\
&=m\left(\bigoplus_{\nu\in P; 0\neq \nu\neq \mu}V_{w(\nu)}\otimes V_{w(\mu-\nu)}\right)\\
&=m\left(\bigoplus_{\nu\in P; 0\neq \nu\neq w(\mu)}V_{\nu}\otimes V_{w(\mu)-\nu}\right).
\end{align*}
Hence we may assume that $\mu$ is the highest weight of $V$. Suppose that 
\[
m\left(\bigoplus_{\nu\in P; 0\neq \nu\neq \mu}V_{\nu}\otimes V_{\mu-\nu}\right)\neq V_{\mu}.
\]
Then,
\begin{align}
 m\left(\bigoplus_{\nu\in P; 0\neq \nu\neq \mu}V_{\nu}\otimes V_{\mu-\nu}\right)=\{0\}\label{eq:assumption}   
\end{align}
since $V_{\mu}$ is one-dimensional. Therefore, the surjectivity of $m$ implies that $m\left(V_{0}\otimes V_{\mu}\right)=V_{\mu}$ or $m\left(V_{\mu}\otimes V_{0}\right)=V_{\mu}$. Assume that  $m\left(V_{0}\otimes V_{\mu}\right)=V_{\mu}$. Then \cref{l:0wtsp} implies that there exist $s\in S$ and $0\neq v_0\in V_0$ such that $m((f_{s}e_{s}\cdot v_0)\otimes v_{\mu})=v_{\mu}$. We have 
\begin{align*}
    0=f_{s}e_{s}\cdot m(v_0\otimes v_{\mu})=m(f_{s}e_{s}\cdot (v_0\otimes v_{\mu}))=m((f_{s}e_{s}\cdot v_0)\otimes v_{\mu}+(e_{s}\cdot v_0)\otimes (f_{s}\cdot v_{\mu})).
\end{align*}
Here the first and third equalities follow from the assumption that $\mu$ is the highest weight of $V$. Therefore, 
\[
m((e_{s}\cdot v_0)\otimes (f_{s}\cdot v_{\mu}))=-v_\mu\neq 0.
\]
In particular, $m(V_{\alpha_s}\otimes V_{\mu-\alpha_s})\neq \{0\}$, which contradicts the equality \eqref{eq:assumption} since $\alpha_s$ is not a dominant integral weight when the rank of $\fg$ is greater than 1. The exactly parallel argument is applicable in the case when $m\left(V_{\mu}\otimes V_{0}\right)=V_{\mu}$. Hence the proof is completed. 
\end{proof}
\begin{thm}\label{t:genmin}
Let $V$ be a quasi-minuscule $\fg$-module. Assume that there exists a surjective homomorphism $m\colon V\otimes V\to V$ of $\fg$-modules satisfying $m(V_0\otimes V_0)=\{0\}$. Then all matrix coefficients of $V$ can be expressed as polynomials in generalized minors. 
\end{thm}
\begin{proof}
It suffices to show that $C^{V}(f, v)$ can be expressed as a polynomial in generalized minors for weight vectors $f\in (V^{\ast})_{\mu}$ and $v\in V_{\nu}$. If $\mu, \nu\in \wt (V)\setminus \{0\}$, then $C^{V}(f, v)$ is a generalized minor up to a constant. Hence we deal with the case when $\mu$ or $\nu$ equals $0$. 

Consider the $\fg$-module homomorphism $m^{\ast}\colon V^{\ast}\to (V\otimes V)^{\ast}\simeq V^{\ast}\otimes V^{\ast}$ induced by $m$. Then, the condition $m(V_0\otimes V_0)=\{0\}$ implies that 
\begin{align}
m^{\ast}((V^{\ast})_0)\subset \bigoplus_{\eta\in \wt (V)\setminus \{0\}}(V^{\ast})_{-\eta}\otimes (V^{\ast})_{\eta}.\label{eq:mdual} 
\end{align}
Assume that $\mu=0$, that is, $f\in (V^{\ast})_0$. Then, by \eqref{eq:mdual}, 
\[
m^{\ast}(f)=\sum_{\eta\in P; \eta\neq 0}f'_{\eta}\otimes f''_{-\eta},
\]
for some $f'_{\eta}\in (V^{\ast})_{-\eta}$ and $f''_{-\eta}\in (V^{\ast})_{\eta}$. 
On the other hand, by \cref{l:Key}, we have 
\[
v=m\left(\sum_{\eta\in P; 0\neq \eta\neq \nu}v'_{\eta}\otimes v''_{\nu-\eta}\right)
\]
for some $v'_{\eta}\in V_{\eta}$ and $v''_{\nu-\eta}\in V_{\nu-\eta}$.
Then 
\begin{align*}
    C^{V}(f, v)&=C^{V}\left(f, m\left(\sum\nolimits_{\eta'\in P: 0\neq \eta'\neq \nu}v'_{\eta'}\otimes v''_{\nu-\eta'}\right)\right)\\
    &=C^{V\otimes V}\left(m^{\ast}(f), \sum\nolimits_{\eta'\in P; 0\neq \eta'\neq \nu}v'_{\eta'}\otimes v''_{\nu-\eta'}\right)\\
    &=C^{V\otimes V}\left(\sum\nolimits_{\eta\in P; \eta\neq 0}f'_{\eta}\otimes f''_{-\eta}, \sum\nolimits_{\eta'\in P; 0\neq \eta'\neq \nu}v'_{\eta'}\otimes v''_{\nu-\eta'}\right)\\
    &=\sum_{\eta\in P; \eta\neq 0}\left(\sum_{\eta'\in P; 0\neq \eta'\neq \nu}C^{V}(f'_{\eta}, v'_{\eta'})C^{V}( f''_{-\eta},  v''_{\nu-\eta'})\right),
\end{align*}
and the matrix coefficients occurring in the right-hand side are all generalized minors up to a constant. Hence we have shown the assertion in this case. 

Assume that $\nu=0$, namely, $v\in V_0$. Then, the condition $m(V_0\otimes V_0)=\{0\}$ implies that 
\[
v=m\left(\sum_{\eta\in P; \eta\neq 0}v'_{\eta}\otimes v''_{-\eta}\right)
\]
for some $v'_{\eta}\in V_{\eta}$ and $v''_{-\eta}\in V_{-\eta}$. Hence, if we write $m^{\ast}(f)$ as $
\sum_{i\in I}f'_{i}\otimes f''_{i}$ 
by some $f'_{i}, f''_{i}\in V^{\ast}$ and a finite index set $I$, then 
\begin{align*}
    C^{V}(f, v)&=C^{V}\left(f, m\left(\sum\nolimits_{\eta\in P; \eta\neq 0}v'_{\eta}\otimes v''_{-\eta}\right)\right)\\
    &=C^{V\otimes V}\left(m^{\ast}(f), \sum\nolimits_{\eta\in P; \eta\neq 0}v'_{\eta}\otimes v''_{-\eta}\right)\\
    &=C^{V\otimes V}\left(\sum\nolimits_{i\in I}f'_{i}\otimes f''_{i}, \sum\nolimits_{\eta\in P; \eta\neq 0}v'_{\eta}\otimes v''_{-\eta}\right)\\
    &=\sum_{i\in I}\left(\sum_{\eta\in P; \eta\neq 0}C^{V}(f'_{i}, v'_{\eta})C^{V}( f''_{i},  v''_{-\eta})\right).
\end{align*} 
We have already shown that the matrix coefficients occurring in the right-hand side can be expressed as polynomials in generalized minors. Hence we completed the proof.
\end{proof}
\begin{thm}\label{t:EG}
When $\fg$ is of type $E_8$ or $G_2$, there exist a quasi-minuscule $\fg$-module $V$ and a surjective homomorphism $m\colon V\otimes V\to V$ of $\fg$-modules satisfying $m(V_0\otimes V_0)=\{0\}$. 
\end{thm}
\begin{proof}
    In the case of type $E_8$, the desired objects are given by the adjoint representation $V=\fg$ together with the Lie bracket $m\colon \fg\otimes \fg\to \fg,\ x\otimes y\mapsto [x,y]$. Note that $V_0=\fg_0=\fh$ in this case. 

    In the case of type $G_2$, the irreducible $\fg$-module  $V=V(\varpi_2)$ is a $7$-dimensional quasi-minuscule $\fg$-module. Fix a highest weight vector $v_{(0,1)}=v_{\varpi_2}$, and set 
    \begin{align*}
        v_{(1,-1)}&\coloneqq f_2\cdot v_{(0,1)},&
        v_{(-1,2)}&\coloneqq f_1\cdot v_{(1,-1)},&
        v_{(0,0)}&\coloneqq f_2\cdot v_{(-1,2)},\\
        v_{(1,-2)}&\coloneqq \frac{1}{2}f_2\cdot v_{(0,0)},&
        v_{(-1,1)}&\coloneqq f_1\cdot v_{(1,-2)},&
        v_{(0,-1)}&\coloneqq f_2\cdot v_{(-1,1)}.
    \end{align*}
    Then these seven vectors form a basis of $V$ such that $\wt v_{(m, n)}=m\varpi_1+n\varpi_2$. Moreover, direct calculation shows that 
    \[
    \widetilde{v}_{(0,1)}:= v_{(0,1)}\otimes v_{(0,0)}- 2v_{(1,-1)}\otimes v_{(-1,2)}+2v_{(-1,2)}\otimes v_{(1,-1)}-v_{(0,0)}\otimes v_{(0,1)}\in V\otimes V
    \]
    satisfies $e_1\cdot \widetilde{v}_{(0,1)}=e_2\cdot \widetilde{v}_{(0,1)}=0$. Hence there exists an injective $\fg$-module homomorphism $\iota\colon V\to V\otimes V$ satisfying $\iota(v_{(0,1)})=\widetilde{v}_{(0,1)}$. Then 
    \begin{align*}
        \iota(v_{(0,0)})&=\iota(f_2f_1f_2\cdot v_{(0,1)})=f_2f_1f_2\cdot\widetilde{v}_{(0,1)}\\
        &=2(v_{(0,1)}\otimes v_{(0,-1)}+v_{(1,-1)}\otimes v_{(-1,1)}-v_{(-1,2)}\otimes v_{(1,-2)}\\
        &\phantom{===}+v_{(1,-2)}\otimes v_{(-1,2)}-v_{(-1,1)}\otimes v_{(1,-1)}-v_{(0,-1)}\otimes v_{(0,1)}).
    \end{align*}
    Therefore, the homomorphism of $\fg$-modules
    \[
    \iota^{\ast}\colon V^{\ast}\otimes V^{\ast}\simeq (V\otimes V)^{\ast}\to V^{\ast}
    \]
    induced by $\iota$ satisfies 
    \[
    \iota^{\ast}((V^{\ast})_0\otimes (V^{\ast})_0)=\{0\}.
    \]
    Hence $V^{\ast}(\simeq V)$ and $m=\iota^{\ast}$ satisfy the desired conditions. 
\end{proof}
\begin{proof}[Proof of \cref{t:mainminor} in the case of type $E_8$ and $G_2$]
Let $\fg$ be a simple Lie algebra of type $E_8$ or $G_2$. Consider the  quasi-minuscule $\fg$-module $V$ stated in \cref{t:EG}. Then the corresponding representation of $G$ is a faithful representation $\rho\colon G\to GL(V)$ which gives a closed immersion $G\to SL(V)$, since the center of $G$ is trivial in this case (see \cite[Lemma 22.1]{Borel}, \cite[Remark 11.2.16]{GW}, \cite[Corollary 1.13]{Brion}). Therefore, the coordinate ring $\C[G]$ is generated by the matrix coefficients of $V$ (see, for example, \cite[Proposition 2.1]{IOS23}), and hence, by the generalized minors thanks to \cref{t:genmin}. 
\end{proof}
\begin{rem}\label{r:F4}
    The Lie algebra $\fg$ of type $G_2$ is realized as the Lie algebra of derivations of the complexified Cayley algebra $\mathfrak{C}$. Hence there is an action of $\mathfrak{g}$ on $\mathfrak{C}$, and $\mathfrak{C}$ is isomorphic to $V(\varpi_2)\oplus V(0)$ as a $\mathfrak{g}$-module. Here the trivial $\fg$-module $V(0)$ corresponds to the space spanned by the unit $\mathbf{1}$ of $\mathfrak{C}$. Then the morphism $m$ in \cref{t:EG} is induced from the multiplication $m_{\mathfrak{C}}$ of $\mathfrak{C}$ as follows:
    \[
    m\colon V(\varpi_2)\otimes V(\varpi_2)\hookrightarrow \mathfrak{C}\otimes \mathfrak{C}\xrightarrow{m_{\mathfrak{C}}}\mathfrak{C}\simeq V(\varpi_2)\oplus V(0)
    \xrightarrow{\text{projection}} V(\varpi_2).
    \]
    
    When $\fg$ is of type $F_4$, there does not exist a homomorphism $m$ of $\fg$-modules satisfying the assumption in \cref{t:genmin}. Indeed, a quasi-minuscule $\fg$-module is isomorphic to the 26-dimensional irreducible module $V(\varpi_4)$ ($\dim V(\varpi_4)_0=2$). We have $V(\varpi_4)\otimes V(\varpi_4)\simeq V(2\varpi_4)\oplus V(\varpi_1)\oplus V(\varpi_3)\oplus V(\varpi_4)$. Therefore, a surjective homomorphism of $\fg$-modules $m\colon V(\varpi_4)\otimes V(\varpi_4)\to V(\varpi_4)$ is uniquely determined up to constant multiplication, and it is the only candidate for the homomorphism $m$ satisfying the assumption in \cref{t:genmin}. However, we have $m(V(\varpi_4)_0\otimes V(\varpi_4)_0)\neq \{0\}$. The Lie algebra $\mathfrak{g}$ of type $F_4$ is realized as the Lie algebra of derivations of the $27$-dimensional exceptional Jordan algebra $\mathfrak{J}$ over $\C$. See, for example, \cite[Section 2]{Yok} for its explicit presentation. Hence there is an action of $\mathfrak{g}$ on $\mathfrak{J}$, and $\mathfrak{J}$ is isomorphic to $V(\varpi_4)\oplus V(0)$ as a $\mathfrak{g}$-module. Here the trivial $\fg$-module $V(0)$ corresponds to the space spanned by the unit $\mathbf{1}$ of $\mathfrak{J}$. Then the multiplication $m_{\mathfrak{J}}$ of $\mathfrak{J}$ induces a surjective homomorphism of $\fg$-modules $m\colon V(\varpi_4)\otimes V(\varpi_4)\to V(\varpi_4)$ as follows:
    \[
    m\colon V(\varpi_4)\otimes V(\varpi_4)\hookrightarrow \mathfrak{J}\otimes \mathfrak{J}\xrightarrow{m_{\mathfrak{J}}}\mathfrak{J}\simeq V(\varpi_4)\oplus V(0)
    \xrightarrow{\text{projection}} V(\varpi_4).
    \]
    However, we can directly check that $m(V(\varpi_4)_0\otimes V(\varpi_4)_0)\neq \{0\}$ since the weight space $V(\varpi_4)_0$ corresponds to the space spanned by the trace-free diagonal matrices in $\mathfrak{J}$ under the presentation in \cite[Section 2]{Yok}. See \cite{Yok} for further details. 
\end{rem}
\bibliographystyle{alpha}

\end{document}